\documentclass[a4paper,11pt,twoside]{amsart}
\usepackage[english]{babel}
\usepackage[utf8]{inputenc}

\usepackage[a4paper,inner=3cm,outer=3cm,top=4cm,bottom=4cm,pdftex]{geometry}
\usepackage{fancyhdr}
\usepackage{fancyhdr}
\usepackage{titlesec}
\usepackage{enumerate}
\usepackage{color}
\usepackage{bold-extra}
\titleformat{\section}{\normalfont\scshape\centering}{\thesection}{1em}{}
\titleformat{\subsection}{\bfseries}{\thesubsection}{1em}{}

\usepackage{color}
\usepackage{bold-extra}
\usepackage{ mathrsfs }
\usepackage{bm}

\usepackage{comment}
\usepackage{graphics}
\usepackage{aliascnt}
\usepackage[pdftex,citecolor=green,linkcolor=red]{hyperref}

\usepackage{amsmath}
\usepackage{amsfonts}
\usepackage{amssymb}
\usepackage{amsthm}
\usepackage{comment}
\usepackage{mathtools}

\newtheorem{theorem}{Theorem}[section]
\newtheorem{corollary}[theorem]{Corollary}

\newtheorem{lemma}[theorem]{Lemma}
\newtheorem{proposition}[theorem]{Proposition}
\theoremstyle{definition}
\newtheorem{definition}[theorem]{Definition}
\newtheorem{remark}[theorem]{Remark}

\numberwithin{equation}{section}

\renewcommand\d{\textnormal{d}}

\newcommand\Z{\mathbb{Z}}

\newcommand{\probP}{\mathscr{P}}

\newcommand{\norm}[1]{\|#1\|}

\DeclareMathOperator{\lcm}{lcm}

\title{Extremal Problems for GCDs and LCMs in Higher Dimensions}
\author{Haozhe Gou}
\address{School of Mathematics, Shandong University, Jinan 250100, China}
\address{D\'epartement de math\'ematiques et de statistique, Universit\'e de Montr\'eal, C.P.~6128, succ.~Centre-ville, Montr\'eal, QC H3C~3J7, Canada}
\email{haozhegou@gmail.com}

\begin{document}

\maketitle
\begin{abstract}
We study extremal problems for tuples of integers chosen from sets $A_i \subset [X_i,2X_i]$ for $1\le i\le k$, under large GCD and small LCM conditions.
For the GCD problem, we extend the work of Green and Walker to higher dimensions. Specifically, for $k\ge 3$, if $\gcd(a_1,\dots,a_k)\ge D$ for at least a proportion $\delta$ of the tuples in $\prod_{i=1}^k A_i$, then
\[
\prod_{i=1}^k |A_i|
\ll_{k,\varepsilon}
\delta^{-k/(k-1)-\varepsilon}
\frac{\prod_{i=1}^k X_i}{D^k}.
\]
The proof is based on a minimal counterexample argument and a new high-dimensional measure concentration lemma. We also establish a large sieve-type inequality to obtain a complementary estimate for the GCD problem.

For the LCM problem, we use a quite different method to show that, for all $k\ge 2$,
\[
\prod_{i=1}^k |A_i|
\ll_{k,\varepsilon}
\delta^{-k/(k-1)}
\frac{L^{k/(k-1)+\varepsilon}}
{\bigl(\prod_{i=1}^k X_i\bigr)^{1/(k-1)}},
\]
whenever $\lcm(a_1,\dots,a_k)\le L$ for at least a proportion $\delta$ of the $k$-tuples in $\prod_{i=1}^k A_i$. Finally, we show that these bounds are essentially best possible up to $\varepsilon$-losses in the exponent.
\end{abstract}

\section{Introduction}
Let $A \subset [X, 2X]$ and $B \subset [Y, 2Y]$ be sets of integers with $1 \le D \le \min(X, Y)$. We are interested in the maximal possible size of the product $|A||B|$ under the hypothesis that a proportion of pairs $(a,b) \in A \times B$ have a large greatest common divisor (GCD). This condition is formulated as follows: for some $\delta \in (0,1]$,
$$ |\left\{ (a,b) \in A \times B : \gcd(a,b) \ge D \right\}| \ge \delta|A||B|.  $$

This problem has recently been solved. Building upon the breakthrough methods developed by Koukoulopoulos and Maynard in \cite{KoukoulopoulosMaynard} in their work on the Duffin--Schaeffer conjecture, Green and Walker \cite{GreenWalker} established a uniform bound for this problem, showing that $|A||B| \ll_{\varepsilon} \delta^{-2-\varepsilon} XY/D^2$ for any $\varepsilon > 0$; this bound is sharp up to $\delta^{-\varepsilon}$. Their proof is highly structural, essentially relying on a delicate inductive argument on the number of prime factors of the elements in $A \cup B$.

A natural direction  is to explore this phenomenon in higher dimensions. This leads us to the central problem of this paper: for $k \ge 3$, what is the optimal upper bound on the product of sizes $\prod_{i=1}^k|A_i|$ for sets $A_i \subset [X_i, 2X_i]$, given that
\begin{equation*}
    |\{ (a_1,\dots,a_k) \in \prod_{i=1}^k A_i : \gcd(a_1,\dots,a_k) \ge D \}| \ge \delta \prod_{i=1}^k|A_i|?
\end{equation*}

Our first main result gives an essentially sharp bound for this high-dimensional GCD
problem.

\subsection{High-Dimensional GCD Problems}
For a set of integers $S$, let $\probP(S)$ denote the set of prime factors of elements of $S$.
Let
\[
\probP_{\mathrm{small}}(S):=\{p\in \probP(S): p\le p_0\},
\]
where $p_0$ is a sufficiently large constant depending on $\varepsilon$ and $k$.

\begin{theorem}[High-Dimensional GCD]\label{thm:main}
    Let $k \geq 3$, $X_1, \dots, X_k, D \in [1, \infty)$ with $D \leq \min(X_1, \dots, X_k)$, and $\delta \in (0,1]$. Let $A_i \subset [X_i, 2X_i]$ ($1 \leq i \leq k$) be integer sets such that $\gcd(a_1,\dots,a_k) \ge D$ for at least $\delta\prod_{i=1}^k|A_i|$ tuples $(a_1,\dots,a_k)$ in $\prod_{i=1}^k A_i$.
    Then for any $\varepsilon \in (0,1)$, we have
    \begin{equation}\label{eq:main result}
            \prod_{i=1}^k|A_i| \leq C_k^{1 + \#\probP_{\text{small}}\left(\bigcup_{i=1}^k A_i\right)}\delta^{-\frac{k+{\varepsilon}/(k-1)}{k-1}}   \frac{\prod_{i=1}^k X_i}{D^k},
    \end{equation}
where $C_k=2^{4k}.$
In particular, the following simplified bound holds:
    \begin{align*}
    \prod_{i=1}^k|A_i| 
    \ll_{\varepsilon,k}\delta^{-\frac{k}{k-1}-{\varepsilon}}   \frac{\prod_{i=1}^k X_i}{D^k}.
    \end{align*}
\end{theorem}

\begin{remark}
First, it is worth noting that while the simplified bound is what we are really interested in, the explicit inequality and precise exponents in \eqref{eq:main result} are essential to reach a contradiction in our proof.
   Second, we have stated our main theorem only for dimensions $k \ge 3$, because of some technical reasons in the high-dimensional measure concentration lemma (see Lemma \ref{lem:measure-concentration} below). However,  for $k=2$,  the bound reads
    \[
    |A_1||A_2| \ll_{\varepsilon} \delta^{-2-\varepsilon} \frac{X_1 X_2}{D^2},
    \]
    which coincides precisely with the bound established by Green and Walker \cite{GreenWalker}. 
\end{remark}
To appreciate the strength of Theorem \ref{thm:main}, we compare it with a trivial counting bound. Let $\Omega \subseteq \prod_{i=1}^k A_i$ be the set of tuples with $\gcd(a_1, \dots, a_k) \ge D$. By assumption, $|\Omega| \ge \delta \prod_{i=1}^k |A_i|$. Summing the number of tuples divisible by each $d \ge D$, we obtain the following trivial upper bound on the size of $\Omega$:
\[
\sum_{d \ge D} \prod_{i=1}^k  \frac{X_i}{d}  \le \left(\prod_{i=1}^k X_i\right) \sum_{d \ge D} \frac{1}{d^k} \ll_k \frac{\prod_{i=1}^k X_i}{D^{k-1}}.
\]
This implies  $\prod_{i=1}^k |A_i| \ll_k \delta^{-1} D^{-(k-1)} \prod_{i=1}^k X_i$. In contrast, Theorem \ref{thm:main} provides a stronger result for large $D$. Moreover, our bound is essentially sharp; the factors $\delta^{-\frac{k}{k-1}}$ and $D^{-k}$ are optimal, as discussed in Section \ref{sec:sharp}.

It is also worth noting that if $\delta$ is extremely small
(roughly, $\delta \ll D^{-(k-1)}$, up to $\varepsilon$-losses), then the right-hand side
in Theorem~\ref{thm:main} may exceed the trivial upper bound
$\prod_{i=1}^k |A_i|\le \prod_{i=1}^k X_i$, so the theorem is most informative when
$\delta$ is not too small.


We also give a second, more elementary approach to the GCD problem, based on a large
sieve-type inequality. This leads to the following complementary estimate.
\begin{theorem}\label{thm:k-dim-hybrid}
Assume the same hypotheses as in Theorem~\ref{thm:main}, with $k\ge 2$.
Then
\begin{equation}\label{eq:hybrid-conditional}
\prod_{i=1}^k |A_i|
\ll_{k,\varepsilon}
\delta^{-\frac{k}{k-1}}
\frac{\prod_{i=1}^k X_i}{D^{k-\varepsilon}}.
\end{equation}
\end{theorem}

Comparing the simplified forms of the two bounds, one sees that
Theorem~\ref{thm:main} is stronger when $\delta$ is not too small
(roughly, when $\delta \gg D^{-1}$),
whereas Theorem~\ref{thm:k-dim-hybrid} may be preferable in the complementary range.

\subsection{Duality and Asymmetry: from GCD to LCM}
It is natural to consider the dual problem concerning the least common multiple (LCM). Intuitively, if a set of integers is arithmetically structured in the sense that many of its tuples have a ``small" LCM, the set itself should be constrained in size. 

Let $\Omega$ be the set of tuples in $\prod_{i=1}^k A_i$ with $\lcm(a_1, \dots, a_k) \le L$, whose size is assumed to be at least $\delta \prod_{i=1}^k |A_i|$. A first trivial estimate is easy to obtain. By relaxing the condition $\lcm(a_1, \dots, a_k)=l$ to the weaker condition $a_i|l$ for all $i$, we have
\begin{equation}\label{eq:LCM trivial}
    \delta \prod_{i=1}^k |A_i| \le |\Omega| \le \sum_{l \le L} \prod_{i=1}^k d_{A_i}(l) \le \sum_{l \le L} \tau(l)^k \ll L(\log L)^{2^k-1},
\end{equation}
where $d_{A_i}(l)$ counts the divisors of $l$ in $A_i$, and $\tau(l)$ is the standard divisor function. This yields the trivial bound  $\prod_{i=1}^k |A_i| \ll \delta^{-1}L(\log L)^{2^k-1}$ for the LCM problem.

When $k=2$, the LCM problem can be reduced to the GCD problem using the identity
\[
\lcm(a,b)=\frac{ab}{\gcd(a,b)}.
\]
Indeed, if $A_1\subset [X_1,2X_1]$ and $A_2\subset [X_2,2X_2]$, then
$\lcm(a_1,a_2)\le L$ implies $\gcd(a_1,a_2) \ge a_1a_2/L \ge X_1X_2/L$, for which the 2-dimensional GCD result is available. So we have the following corollary.

\begin{corollary}[2-Dimensional LCM]
  Let $X_1,  X_2, L \in [1, \infty)$ with $L \ge \max(X_1, X_2)$, and $\delta \in (0,1]$. Let $A_i \subset [X_i, 2X_i]$ ($i=1,2$) be integer sets such that $\lcm(a_1,a_2) \le L$ for at least $\delta|A_1||A_2|$ tuples $(a_1,a_2)$ in $A_1\times A_2$.
    Then for any $\varepsilon \in (0,1)$, we have
    \begin{equation*}
            |A_1||A_2| \ll_{\varepsilon} \delta^{-2-{\varepsilon}}   \frac{L^2}{X_1X_2}.
    \end{equation*}
\end{corollary}
However, this simple path is closed for dimensions $k \ge 3$. The relationship between the LCM and GCD becomes more complex, as it is governed by the Principle of Inclusion-Exclusion.
For instance, one always has
\[
\gcd(a_1,\cdots,a_k)=\frac{\prod_{i=1}^ka_i}{\lcm\left(\prod_{i\ne 1}a_i,\cdots,\prod_{i\ne k}a_i \right)}\ge \frac{\prod_{i=1}^kX_i}{\lcm(a_1,\cdots,a_k)^{k-1}}.
\]
By using Theorem \ref{thm:main}, together with the condition $\lcm(a_1,\cdots,a_k)\le L$ for at least $\delta$ proportion of tuples in $\prod_{i=1}^kA_i$, this yields only a very weak bound
\[
\prod_{i=1}^k|A_i|\ll \delta^{-\frac{k}{k-1}-\varepsilon} \frac{L^{k(k-1)}}{\big(\prod_{i=1}^kX_i\big)^{k-1}}.
\]
Comparing this with the trivial bound $\prod_{i=1}^k |A_i| \ll \delta^{-1} L (\log L)^{2^k-1}$ from \eqref{eq:LCM trivial}, we see that this result is superior only in the very narrow regime where $L$ remains close to the scale of $X_i$ (roughly when $L$ lies in the range $[X, X^{1+c_{k}}]$ for some small constant $c_k$, if $X_i \asymp X$ for all $i$).

Nevertheless, by adapting the minimal counterexample method used for the GCD problem, one can also show that
for $k\ge 3$,
\[
\prod_{i=1}^k|A_i|
\ll_{\varepsilon,k}
\delta^{-\frac{k}{k-1}-\varepsilon}
\frac{L^k}{\prod_{i=1}^k X_i},
\]
which may be viewed as a dual analogue of Theorem~\ref{thm:main}.

Our main result for the LCM problem is the following stronger estimate.
It is proved by a direct counting method,
which in this setting turns out to be both simpler and stronger.

\begin{theorem}\label{thm:lcm-dim-clean}
Let $k \ge 2$, $X_1,\dots,X_k,L\in [1,\infty)$ with
$L\ge \max(X_1,\dots,X_k)$, and let $\delta\in(0,1]$.
Let $A_i\subset [X_i,2X_i]$ $(1\le i\le k)$ be integer sets such that
$\lcm(a_1,\dots,a_k)\le L$ for at least
$\delta \prod_{i=1}^k |A_i|$ tuples $(a_1,\dots,a_k)$ in $\prod_{i=1}^k A_i$.
Then for every fixed $\varepsilon\in(0,1)$,
\[
\prod_{i=1}^k |A_i|
\ll_{k,\varepsilon}
\delta^{-\frac{k}{k-1}}
\frac{L^{\frac{k}{k-1}+\varepsilon}}
{\left(\prod_{i=1}^k X_i\right)^{\frac{1}{k-1}}}.
\]
\end{theorem}

    For the GCD problem, the  bounds \eqref{eq:hybrid-conditional} and \eqref{eq:main result}  obtained via the two different methods—Theorem \ref{thm:k-dim-hybrid} and Theorem \ref{thm:main}—are consistent (up to $\varepsilon$ losses in the exponents). 
In contrast, for the LCM problem,
the direct counting argument in Theorem~\ref{thm:lcm-dim-clean} gives a substantially
stronger bound. As detailed in Section \ref{sec:sharp}, this is uniformly optimal (up to $\varepsilon$ losses in the exponents). 

Even so, comparing Theorem \ref{thm:lcm-dim-clean} with Theorem \ref{thm:main}, it would be interesting to obtain the following result for the LCM problem, that is,  an $\varepsilon$ loss only in the exponent of $\delta$:
  \[
    \prod_{i=1}^k |A_i| \ll_{k,\varepsilon} \delta^{-\frac{k}{k-1}-\varepsilon} \frac{L^{\frac{k}{k-1}}}{\big(\prod_{i=1}^k X_i\big)^{\frac{1}{k-1}}}
    \]
    holds for all $k\ge 3$ and uniformly for all parameters $X_i, L$ and $\delta$.

\subsection{Notation.}
We use standard asymptotic notation throughout.
The relation $f \ll g$ (equivalently $f = O(g)$) means $|f| \leq C g$ for some constant $C > 0$.
We write $f \asymp g$ if both $f \ll g$ and $g \ll f$ hold.
If the implied constants depend on a parameter, we indicate this with a subscript. 
For brevity, we do not always explicitly indicate this dependence when it is clear from the context.

Denote  the $p$-adic valuation by $v_p(n)$, the highest power of $p$ dividing $n$ (extended to rationals by $v_p(a/b) = v_p(a) - v_p(b)$). 

For any real number $q \ge 1$, we denote by $\ell^q(\mathbb{Z})$ the space of sequences $x = (x_n)_{n \in \mathbb{Z}}$ for which the norm $\norm{x}_{\ell^q(\mathbb{Z})} := (\sum_{n \in \mathbb{Z}} |x_n|^q)^{1/q}$ is finite. A finitely supported probability measure $\mu$ on $\mathbb{Z}^k$ is understood as a non-negative function $\mu: \mathbb{Z}^k \to [0,1]$ such that $\sum_{\bm{t} \in \mathbb{Z}^k} \mu(\bm{t}) = 1$ and $\mu(\bm{t}) = 0$ for all but finitely many $\bm{t} \in \mathbb{Z}^k$.

\section{High-Dimensional GCD Problem}
In this section, we establish our main result, which generalizes the theorem of Green and Walker \cite{GreenWalker} to higher dimensions. Since there are some differences in the higher-dimensional cases, for technical reasons, we will restrict $k\ge3$. 

The proof of Theorem \ref{thm:main} essentially follows the structural induction method of Green and Walker; however, here we need to establish a high-dimensional measure concentration lemma.

\subsection{A High-Dimensional Measure Concentration Lemma}

The core of our structural argument relies on a new high-dimensional measure concentration lemma. We state and prove this lemma in this section; its general nature may be of independent interest.

Let $\bm t=(t_1,\dots,t_k)$ denote a vector in $\Z^k$. We define the ``GCD norm'' as $$\norm{\bm t}_{\mathrm{GCD}} := \sum_{i=1}^k t_i - k\min_{1\le i\le k}(t_i).$$

\begin{lemma}\label{lem:measure-concentration}
   Let $k\ge 3$. Let $q = q(k,\varepsilon) > \frac{k}{k-1}$ and $q' = \frac{q}{q-1}$.  Let $\lambda(k) \in (0,1)$ be a small constant depending on $k$, and assume $0 < \lambda \le \lambda(k)$. Let $c \le 1$.
   Suppose $\mu$ is a finitely supported probability measure on $\Z^k$, and there exist non-negative sequences $x^{(1)}, \dots, x^{(k)} \in \ell^{q'}(\Z)$ with $\norm{x^{(i)}}_{\ell^{q'}(\Z)} = 1$ such that for all $\bm t=(t_1, \dots, t_k) \in \Z^k$,
   \begin{equation} \label{eq:measure upper bound}
       \mu(\bm t) \le c \lambda^{\norm{\bm t}_{\mathrm{GCD}}} \prod_{i=1}^k x^{(i)}_{t_i}.
   \end{equation}
   Then we have

   \begin{enumerate}
       \item {Lower bound on $c$:}
       \begin{equation}\label{eq:lower bound of c}
       c \ge \frac{1}{k}\left(1-\lambda(k) \right)^{k-1}.
       \end{equation}

       \item Concentration:
       There exists an integer $m \in \mathbb{Z}$ such that $\mu$ is highly concentrated on the set $\mathcal{S}_m:= \{ m\bm{1} \} \cup \{ m\bm{1} + \bm{e}_j\mid 1\le j\le k\}$,
where $\bm {1} = (1, \dots, 1)$ and $\bm{e}_j$ denotes the $j$-th standard basis vector in $\mathbb{Z}^k$. Specifically, the measure outside this set satisfies
       \begin{equation}\label{eq:measure concentration}
            \sum_{\bm t \notin \mathcal{S}_m} \mu(\bm t) \ll_{k} \lambda^{\eta},
       \end{equation}
       where $\eta = \eta(k,q) := \min \{ k+q-2, \, k(q-1), \, 2q-1, \, q+1 \}$.
   \end{enumerate}
\end{lemma}

\begin{proof}
We begin with the fact that $\mu$ is a probability measure, its total mass is 1. Keep in mind that $0\le x_t^{(i)}\le 1$ for all $1\le i\le k$ and $t\in\Z$.
Applying the upper bound from the hypothesis \eqref{eq:measure upper bound}, we have
\begin{align*}
    1 = \sum_{\bm t \in \Z^k} \mu(\bm t) &\le c \sum_{t_1, \dots, t_k \in \Z} \lambda^{\sum_{i=1}^k |t_i - \min(t_1,\cdots,t_k)|} \prod_{i=1}^k x^{(i)}_{t_i}.
\end{align*}
Let $t_j = \min(t_1,\cdots,t_k)$ be the minimum value for a given tuple (if not unique, pick one such index $j$). We sum over all possible locations $j$  to cover all  tuples in $\Z^k$. Making the change of variables $t = t_j$ and $\ell_i = t_i - t \ge 0$ (so $\ell_j = 0$), we obtain$$1 \le c \sum_{j=1}^k \sum_{\ell_1, \dots, \ell_k \ge 0 \atop \ell_j = 0} \lambda^{\sum_{i \ne j} \ell_i} \left( \sum_{t \in \mathbb{Z}} \prod_{i=1}^k x^{(i)}_{t+\ell_i} \right).$$
By applying the generalized H\"older inequality to the inner sum over $t$, and by the translation invariance of $\ell^p$-norms, we have
$$\sum_{t \in \mathbb{Z}} \prod_{i=1}^k x^{(i)}_{t+\ell_i} \le \prod_{i=1}^k \|x^{(i)}\|_{\ell^k}.$$Since $q > \frac{k}{k-1}$ implies the conjugate exponent satisfies $q' < k$, the standard embedding $\ell^{q'} \subset \ell^k$ gives $\|x^{(i)}\|_{\ell^k} \le \|x^{(i)}\|_{\ell^{q'}}$. Given the hypothesis $\|x^{(i)}\|_{\ell^{q'}} = 1$, the inner sum is therefore bounded by $1$.

The remaining sum over $\ell_i$ factors into a product of $k-1$ geometric series
\[
\sum_{j=1}^k \sum_{\substack{\ell_i \in \Z_{\ge 0} \\ i \ne j}} \lambda^{\sum_{i \ne j} \ell_i} = \sum_{j=1}^k \prod_{i \ne j} \left( \sum_{\ell_i=0}^\infty \lambda^{\ell_i} \right) 
= k \left( \frac{1}{1-\lambda} \right)^{k-1}.
\]
Combining these results yields the desired lower bound for $c$. This completes the proof of the first part of the lemma.

Turning to the second part, let $\sup_{t\in\Z}\prod_{1\le i\le k}x^{(i)}_t=1-\gamma$ for some $\gamma\in [0,1]$, and suppose this supremum is attained at $t=m$. Then for each $i\in\{1,\dots,k\}$, we have $x^{(i)}_m\ge 1-\gamma$. Using the assumption $\|x^{(i)}\|_{\ell^{q'}} = 1$, we have
$$1=(x^{(i)}_m)^{q'}+\sum_{t\ne m}(x^{(i)}_t)^{q'},$$
which implies
\begin{equation}\label{eqn:pointwise bound}
    \sum_{t\ne m}(x^{(i)}_t)^{q'}\le q'\gamma\ll_k \gamma, \quad \text{and} \quad x^{(i)}_t\ll_k \gamma^{1/q'} \ \text{when} \ t\ne m.
\end{equation}

For any set $S\subset \Z^k$, define $\mu(S):=\sum_{\bm t \in S}\mu(\bm t)$. We partition the space $\Z^k$ into three disjoint subsets based on the distance $d_m(\bm t) := \sum_{i=1}^k |t_i - m|$. Let $S_0$ and $S_1$ be the sets where $d_m(\bm t)=0$ and $1$ respectively, and let $S_2$ denote the region where $d_m(\bm t) \ge 2$. We now bound the contribution $\mu(S_j)$ of each part for $j=0,1,2$.

For $j=0$, the center consists of the single point $S_0 = \{ (m, \dots, m) \}$.

For $j=1$, we consider the ``near center'' region $S_1$. Every tuple in $S_1$ has exactly one index $\ell \in \{1,\dots,k\}$ such that $t_{\ell}=m-1$ or $m+1$, and all other $t_i$ are equal to $m$. 
If $t_{\ell}=m-1$, we have $\norm{\bm t}_{\mathrm{GCD}}=(k-1)m+(m-1)-k(m-1)=k-1$.
Then by the trivial bound $x^{(i)}_m\le 1$ and the pointwise bound \eqref{eqn:pointwise bound} for $t_\ell\ne m$, we have
\begin{equation}\label{eq:S1-}
\begin{aligned}
      \mu (S_{1,-})&:=\sum_{\substack{(t_1,\cdots,t_k)\in S_1\\t_{\ell}=m-1}}\mu(t_1,\cdots,t_k)\le \lambda^{k-1}\sum_{(t_1,\cdots,t_k)\in S_1}x^{(\ell)}_{t_{\ell}}\prod_{i\ne \ell}x^{(i)}_{t_i}\\
      &\ll \lambda^{k-1} \gamma^{1/q'}\sum_{(t_1,\cdots,t_k)\in S_1}1 \ll_k\lambda^{k-1} \gamma^{1/q'}.
\end{aligned}
\end{equation}
 If $t_{\ell}=m+1$, we have $\norm{\bm t}_{\mathrm{GCD}}=1$.
In this case 
\begin{equation}\label{eq:S1+}
\begin{aligned}
      \mu (S_{1,+}):=\sum_{\substack{(t_1,\cdots,t_k)\in S_1\\t_\ell=m+1}}\mu(t_1,\cdots,t_k)\le \lambda\sum_{(t_1,\cdots,t_k)\in S_1}x^{(\ell)}_{t_{\ell}}\prod_{i\ne \ell}x^{(i)}_{t_i}\ll_k\lambda \gamma^{1/q'}.
\end{aligned}
\end{equation}
Combining \eqref{eq:S1-} and \eqref{eq:S1+} gives 
\begin{align}\label{eq:S1}
    \mu(S_1)=\mu(S_{1,-})+\mu(S_{1,+})\ll_k \lambda\gamma^{1/q'}.
\end{align}

For $j=2$, we partition the ``far-from-center'' set $S_{ 2}$ into two disjoint subsets:
\begin{itemize}
    \item $S_{2a}$: The set of tuples where at least two coordinates are different from $m$.
    \item $S_{2b}$: The set of tuples where exactly one coordinate is different from $m$, and its distance from $m$ is at least 2.
\end{itemize}
We further decompose $S_{2a}$ into diagonal tuples $S^{\text{diag}}_{2a}:=\{(t,\dots,t) \mid t\ne m \}$ and non-diagonal tuples $S^{\text{non-diag}}_{2a}:=S_{2a}\setminus S_{2a}^{\text{diag}}$. The total measure is
$$\mu(S_{\ge 2}) = \mu(S_{2a}^{\text{diag}}) + \mu(S_{2a}^{\text{non-diag}}) + \mu(S_{2b}).$$
We bound each part separately.

\noindent \textbf{Bounding the measure of $S_{2a}^{\text{diag}}$.}
Since $q'/k<1$. By \eqref{eqn:pointwise bound} and H\"older's inequality,
\begin{equation}\label{eq:S2a-diag}
\begin{aligned}
    \mu(S_{2a}^{\text{diag}})&\le\sum_{t\ne m}\prod_{i=1}^{k}x^{(i)}_t
    \le \prod_{i=1}^k \| x^{(i)}_{t\ne m} \|_{\ell^k} \le \prod_{i=1}^k \| x^{(i)}_{t\ne m} \|_{\ell^{q'}} \ll \gamma^{k/q'}.
\end{aligned}
\end{equation}
\textbf{Bounding the measure of $S_{2a}^{\text{non-diag}}$.}
A tuple $(t_1, \dots, t_k)$ in $S_{2a}^{\text{non-diag}}$ has at least two coordinates different from $m$, and not all coordinates are equal.  For each such tuple, let \(t_{\min}=\min_{1\le i\le k}(t_i)\) and \(t_{\max}=\max_{1\le i\le k}(t_i)\) so that $t_{\min} \ne t_{\max}$. Let $r, s$ be distinct indices such that $t_{r}, t_s \ne m$. By the pointwise bound in \eqref{eqn:pointwise bound}, we have 
\[
\prod_{i=1}^k x^{(i)}_{t_i}=\prod_{i=1}^k (x^{(i)}_{t_i})^{q'/k}\cdot(x^{(r)}_{t_r}x^{(s)}_{t_s})^{1-q'/k} \prod_{i\ne r,s} (x^{(i)}_{t_i})^{1-q'/k}  \ll_k  \gamma^{\frac{2}{q'}(1-\frac{q'}{k})}\prod_{i=1}^k (x^{(i)}_{t_i})^{q'/k}. 
\]
Then we bound the measure of $S_{2a}^{\text{non-diag}}$ by summing over all non-diagonal tuples that satisfy these two conditions, i.e.,
\begin{equation*}
\begin{aligned}
    \mu(S_{2a}^{\text{non-diag}}) 
    & \ll_k\gamma^{\frac{2}{q'}(1-\frac{q'}{k})} \sum_{\substack{(t_1, \dots, t_k) \in \Z^k \\ t_{\min} \ne t_{\max}\\ \exists r,s,\ \text{s.t.}\ t_r, t_s\ne m}} \lambda^{\sum_{i=1}^k |t_i - t_{\min}|} \prod_{i=1}^k (x^{(i)}_{t_i})^{q'/k}.
    \end{aligned}
\end{equation*}
Let $\Sigma$ denote the summation above. Following the method of the first part, we sum over all possible locations $j$ of the minimum value and perform a change of variables $\ell_i = t_i - t_j \ge 0$. 
The condition $t_{\min} \ne t_{\max}$ implies that not all $\ell_i$  are zero. The condition that at least two coordinates are not $m$ implies either the minimum $t_j \ne m$ and there is at least one other index $r \ne j$ such that $t_j + \ell_r \ne m$, or there are at least two indices $r, s \ne j$ such that $t_j + \ell_r \ne m$ and $t_j + \ell_s \ne m$. 
Thus, we can bound $\Sigma$ by splitting the sum over $t_j$:
\begin{align*}
    \Sigma &\le \sum_{j=1}^k \Bigg(\sum_{\substack{t_j \ne m }} \sum_{\substack{\ell_i \in \Z_{\ge 0} , i \ne j \\ \text{not all } \ell_i=0\\\exists r\ne j\ \text{s.t.}\ t_j+\ell_r\ne m}} \lambda^{\sum_{i \ne j} \ell_i} \prod_{i=1}^k \left(x^{(i)}_{t_j+\ell_i}\right)^{q'/k}\\
    &\qquad\qquad + \sum_{t_j \in \Z} \sum_{\substack{\ell_i \in \Z_{\ge 0} , i \ne j \\ \text{not all } \ell_i=0\\\exists r,s\ne j\ \text{s.t.}\ \ell_r,\ell_s\ne m-t_j}} \lambda^{\sum_{i \ne j} \ell_i} \prod_{i=1}^k \left(x^{(i)}_{t_j+\ell_i}\right)^{q'/k}\Bigg).
\end{align*}
Interchanging the summation and taking the supremum over the $\ell_i$ terms gives
\begin{align*}
    \Sigma \le  &\sup_{\ell_i\in\Z,i\ne j} \Bigg( \sum_{\substack{t_j \ne m\\\exists r\ne j, \text{s.t.} t_j+\ell_r\ne m}} \prod_{\substack{i=1\\}}^k \left(x^{(i)}_{t_j+\ell_i}\right)^{q'/k} + \sum_{\substack{t_j \in \Z\\\exists r,s\ne j, \text{s.t.} \ell_r,\ell_s\ne m-t_j}} \prod_{\substack{i=1}}^k \left(x^{(i)}_{t_j+\ell_i}\right)^{q'/k} \Bigg)\\
    & \times\sum_{j=1}^k \Bigg( \sum_{\substack{\ell_i \in \Z_{\ge 0},\ i \ne j \\ \text{not all } \ell_i=0}} \lambda^{\sum_{i \ne j} \ell_i} \Bigg).
\end{align*}
The geometric series term here is bounded by
\begin{align*}
   \Bigg( \sum_{\ell_i \in \Z_{\ge 0} , i \ne j} \prod_{i \ne j} \lambda^{\ell_i} \Bigg) - \lambda^0 \
    = \left( \prod_{i \ne j} \sum_{\ell_i=0}^\infty \lambda^{\ell_i} \right) - 1 
    = \left( \frac{1}{1-\lambda} \right)^{k-1} - 1\ll_k\lambda,
\end{align*} 
For the supremum term, since at least two terms in the product correspond to indices where the value is not $m$. By H\"older's inequality and \eqref{eqn:pointwise bound},  
we have
\begin{align*}
     &\sum_{\substack{t_j \ne m\\\exists r\ne j, \text{s.t.} t_j+\ell_r\ne m}} \prod_{\substack{i=1\\}}^k \left(x^{(i)}_{t_j+\ell_i}\right)^{q'/k} \le \\
     &\qquad\bigg(\sum_{t_j\ne m}(x^{(j)}_{t_j})^{q'}\Big)^{1/k}\bigg(\sum_{t_j+\ell_r\ne m}(x^{(r)}_{t_j+\ell_r})^{q'}\bigg)^{1/k} \prod_{i\ne j,r}\bigg(\sum_{t_i\in\Z}(x^{(i)}_{t_i})^{q'}\bigg)^{1/k}\ll \gamma^{\frac{2}{k}}.
\end{align*}
Similarly, in the second case ($\exists r,s\ne j, \text{s.t.} \ell_r,\ell_s\ne m-t_j$), the shifting guarantees that there are at leat two terms  avoid $m$, yielding the same bound $\ll_k \gamma^{2/k}$.
Combining these bounds,  we obtain
\begin{equation}\label{eq:S2a-nondiag}
    \mu(S_{2a}^{\text{non-diag}}) \ll_k   \gamma^{\frac{2}{q'}(1-\frac{q'}{k})}\lambda \gamma^{\frac{2}{k}}\ll_k \lambda \gamma^{{2}/{q'}}.
\end{equation}

\noindent \textbf{Bounding the measure of $S_{2b}$.}
For tuples in $S_{2b}$, exactly one coordinate $t_{\ell} \ne m$ with $|t_{\ell}-m| = s \ge 2$.
The product term is bounded by $\prod x^{(i)}_{t_i} \le x^{(\ell)}_{t_{\ell}} \le \gamma^{1/q'}$.
Note that $\norm{\bm t}_{\mathrm{GCD}}$ equals  $s$ 
if $t_{\ell}=m+s$, and equals $(k-1)s$ if $t_{\ell}=m-s$.  We bound crudely $\lambda^{\norm{\bm t}_{\mathrm{GCD}}}\le\lambda^{s}$. Thus,
\begin{equation}\label{eq:S2b}
\begin{aligned}
\mu(S_{2b}) &\le \gamma^{1/q'} \sum_{\ell=1}^k \sum_{s\ge 2} \lambda^s \sum_{|t_{\ell}-m|=s} 1 
\ll_k \lambda^2 \gamma^{1/q'},
\end{aligned}
\end{equation}
where  $\sum_{s\ge 2}\lambda^s = \frac{\lambda^2}{1-\lambda} \ll_k \lambda^2$ holds, since $\lambda$ is bounded away from 1 (specifically $\lambda \le \lambda(k) < 1$).

Finally, to reveal the relation between $\gamma$ and $\lambda$, we need to bound the measure of $S_0$ and $S_{2a}^{\text{diag}}$ together. In this set 
$\norm{\bm t}_{\mathrm{GCD}}=0$. Then by H\"older's inequality again,
\begin{equation}\label{eqn:S0+S2a-diag}
\begin{aligned}
   \mu(S_0 \cup S_{2a}^{\text{diag}}) &\le \sum_{t \in \Z} \prod_{i=1}^k x^{(i)}_t \le \sup_{t \in \Z} \left(\prod_{i=1}^k x^{(i)}_t\right)^{1-q'/k} \sum_{t \in \Z} \left(\prod_{i=1}^k x^{(i)}_t\right)^{q'/k}\\
   &= (1-\gamma)^{1-q'/k} \prod_{i=1}^k \left( \sum_{t \in \Z}(x^{(i)}_t)^{q'}\right)^{1/k}\\
   &\le 1 - (1-q'/k)\gamma.
\end{aligned}
\end{equation}
Putting  \eqref{eqn:S0+S2a-diag}, \eqref{eq:S1}, \eqref{eq:S2a-nondiag} and \eqref{eq:S2b} all together we have
\begin{align*}
1&=\mu(S_0\bigcup S_{2a}^{\text{diag}})+\mu(S_1)+ \mu(S_{2a}^{\text{non-diag}})+\mu(S_{2b})\\
& \le  1-(1-q'/k)\gamma +O_k\big(\lambda \gamma^{1/q'}+\lambda \gamma^{{2}/{q'}}+\lambda^2 \gamma^{1/q'}\big).
\end{align*}
This implies $\gamma \ll_k \lambda \gamma^{1/q'}$, so  $\gamma \ll_k \lambda^q$ (since $1-1/q' = 1/q$).
Now observe that
\(
\mathcal S_m = S_0 \cup S_{1,+}.
\)
Hence
\[
\mathbb Z^k\setminus \mathcal S_m
= S_{1,-}\cup S_{2a}^{\mathrm{diag}}\cup S_{2a}^{\mathrm{non\text{-}diag}}\cup S_{2b},
\]
so that
\begin{align*}
\sum_{\bm t\notin \mathcal S_m}\mu(\bm t)
&= \mu(S_{1,-})+\mu(S_{2a}^{\mathrm{diag}})
+\mu(S_{2a}^{\mathrm{non\text{-}diag}})+\mu(S_{2b}) \\
&\ll_k \lambda^{k-1}\gamma^{1/q'}
+ \gamma^{k/q'}
+ \lambda\gamma^{2/q'}
+ \lambda^2\gamma^{1/q'} \\
&\ll_k \lambda^\eta,
\end{align*}
where $\eta = \min(k+q-2, k(q-1), 2q-1, q+1)$.
This completes the proof of Lemma \ref{lem:measure-concentration}.
\end{proof}

\subsection{Structure of a Minimal Counterexample}

\begin{definition}[Minimal Counterexample]
A tuple $(A_1, \dots, A_k)$ is said to be a {\it minimal counterexample} to Theorem \ref{thm:main} if it satisfies the hypotheses of Theorem \ref{thm:main} but violates the bound \eqref{eq:main result}, and such that no counterexample exists whose set of prime divisors is contained in $\probP' \subsetneq \probP(\bigcup_{i=1}^k A_i)$.
\end{definition}

Our first goal is to demonstrate that such a minimal counterexample must possess a strong  arithmetic structure.

\begin{proposition}\label{prop:minimal-structure}
    Suppose $(A_1, \dots, A_k)$ is a minimal counterexample to Theorem \ref{thm:main}. Let $\Omega \subseteq \prod_{i=1}^k A_i$ be the set of integer tuples $(a_1, \dots, a_k)$ such that $\gcd(a_1, \dots, a_k) \ge D$, and $|\Omega| \ge \delta \prod_{i=1}^k |A_i|$. Then there exists an integer $N$ and a subset $\Omega' \subset \Omega$ with $|\Omega'| \geq \frac{1}{2}|\Omega|$ such that for all primes $p$ and all tuples $(a_1, \dots, a_k) \in \Omega'$, 
   the $p$-adic valuations $v_p(a_i)$ are equal to $v_p(N)$ for all $i=1, \dots, k$, with at most one exception which equals $v_p(N)+1$.
\end{proposition}

\begin{proof}
The proof proceeds in three steps.

\textbf{1. Localization at a Prime.}
Fix an arbitrary prime $p \in \probP(\bigcup_{i=1}^k A_i)$. For each set $A_i$, we partition it according to the $p$-adic valuation. Define
$$ A_{i,t} = \{a \in A_i \mid v_p(a)=t\}, \quad \alpha_{i,t} = \frac{|A_{i,t}|}{|A_i|}. $$
Let
$$ \mu_p(t_1, \dots, t_k) = \frac{|\Omega \cap (A_{1,t_1} \times \dots \times A_{k,t_k})|}{|\Omega|}. $$
This is a probability measure on $\Z_{\ge 0}^k$.

For any tuple $(t_1, \dots, t_k)$ for which the set $\Omega \cap (\prod_i A_{i,t_i})$ is non-empty, we define the ``prime-removed'' sets 
$$\bar{A}_{i,t} = p^{-t}A_{i,t}=\{p^{-t}a|a\in A_{i,t} \}.$$
An element $x \in \bar{A}_{i,t}$ lies in the interval $[X_i/p^t, 2X_i/p^t]$.
Observe that the map $a_i \mapsto x_i=p^{-t_i}a_i$ induces a bijection between $\prod_{i=1}^k A_{i,t_i}$ and $\prod_{i=1}^k \bar{A}_{i,t_i}$. Additionally, if $(a_1, \dots, a_k) \in \Omega$, then the corresponding tuple $(x_1, \dots, x_k)$ satisfies
$$\gcd(x_1, \dots, x_k)=\gcd(p^{-t_1}a_1, \dots, p^{-t_k}a_k) =\frac{\gcd(a_1, \dots, a_k)}{p^{\min(t_1, \dots, t_k)}} \ge \frac{D}{p^{\min(t_1, \dots, t_k)}}. $$
Then by using $|\bar{A}_{i,t_i}| = |A_{i,t_i}| = \alpha_{i,t_i}|A_i|$, we get the proportion of elements $(x_1, \dots, x_k)$ in $\prod_{i=1}^k \bar{A}_{i,t_i}$ such that $\gcd(x_1, \dots, x_k)\ge p^{-\min(t_1, \dots, t_k)} D$ is
\begin{align*}
\frac{|\Omega \cap (A_{1,t_1} \times \dots \times A_{k,t_k})|}{\prod_{i=1}^k|\bar{A}_{i,t_i}|}&=\frac{|\Omega \cap (A_{1,t_1} \times \dots \times A_{k,t_k})|}{|\Omega|}\cdot\frac{\prod_{i=1}^k|{A}_{i}|}{\prod_{i=1}^k|{A}_{i,t_i}|}\cdot\frac{|\Omega|}{\prod_{i=1}^k|{A}_{i}|}\\
&\ge \frac{ \mu_p(t_1, \dots, t_k)}{\prod_{i=1}^k \alpha_{i,t_i}}\delta.
\end{align*}
The set of primes dividing $\cup \bar{A}_{i,t_i}$ is a proper subset of $\probP(\cup A_i)$, as $p$ has been removed. By the minimality assumption, the tuple $(\bar{A}_{1,t_1}, \dots, \bar{A}_{k,t_k})$ is not a counterexample. It must therefore satisfy the bound of Theorem \ref{thm:main}, which gives
\begin{equation}\label{eq:not counterexample}
\prod_{i=1}^k |\bar{A}_{i,t_i}|=\prod_{i=1}^k\alpha_{i,t_i}|A_i| \le C_k^{1+\#\probP_{\text{small}}(\cup\bar{A}_{i,t_i})} \left( \frac{ \mu_p(t_1, \dots, t_k)}{\prod_{i=1}^k \alpha_{i,t_i}}\delta\right)^{-\frac{k+\frac{\varepsilon}{k-1}}{k-1}} \frac{\prod (X_i/p^{t_i})}{(D/p^{\min(t_i)})^k}. 
\end{equation}
On the other hand, since  $(A_1, \dots, A_k)$ is a counterexample, we have
\begin{align}\label{eq:counterexample}
    \prod_{i=1}^k |A_i| > C_k^{1 + \#\probP_{\text{small}}\left(\bigcup A_i\right)} \cdot \delta^{-\frac{k+\frac{\varepsilon}{k-1}}{k-1}} \cdot \frac{\prod_{i=1}^k X_i}{D^k}
\end{align}
Note $\probP\left(\bigcup \bar{A}_{i,t_i}\right) \subset \left(\bigcup A_i\right) \setminus \{p\}$, so $\#\probP_{\text{small}}\left(\bigcup \bar{A}_{i,t_i}\right) \leq \#\probP_{\text{small}}\left(\bigcup A_i\right) - \textbf{1}_{p \leq p_0}$.
Comparing \eqref{eq:not counterexample} and \eqref{eq:counterexample} gives  the key measure constraint
\begin{equation}
\begin{aligned}\label{eq:measure_constraint}
    \mu_p(t_1, \dots, t_k) &\le \left(C_k \right)^{-\textbf{1}_{p \le p_0}/q}   \left(p^{-1/q}\right)^{\sum_{i=1}^k t_i-k\min_{1\le i\le k}(t_i)}\prod_{i=1}^k\left(\alpha_{i,t_i}\right)^{(q-1)/q}  \\
    &\le 2^{-2k\cdot \textbf{1}_{p \le p_0}}   \left(p^{-1/q}\right)^{\norm{\bm t}_{\mathrm{GCD}}}\prod_{i=1}^k\left(\alpha_{i,t_i}\right)^{1/q'}, 
\end{aligned}
\end{equation}
since $C_k=2^{4k}$, $q=\frac{k+\frac{\varepsilon}{k-1}}{k-1}\le 2$ (for $k\ge 3$) and $q' = \frac{q}{q-1}$.

\textbf{2. Applying the Measure Concentration Lemma.}
The inequality \eqref{eq:measure_constraint} fits the hypothesis of the high-dimensional measure concentration lemma (Lemma \ref{lem:measure-concentration}). With 
$$ c =2^{-2k\cdot \textbf{1}_{p \le p_0}},\ \lambda = p^{-1/q},\ \lambda(k)= 2^{-\frac{k-1}{k+1}}, \ \text{and}\ x^{(i)}_{t_i} = (\alpha_{i,t_i})^{1/q'},$$ 
the first part of Lemma \ref{lem:measure-concentration}  implies that
\[
c =2^{-2k\cdot \textbf{1}_{p \le p_0}} \ge \frac{1}{k}\left(1-2^{-\frac{k-1}{k+1}}\right)^{k-1},
\]
which forces every prime $p$ in $\probP(\bigcup_{i=1}^k A_i)$ to satisfy $p>p_0$.
Now, for each such prime $p > p_0$, the second part of Lemma \ref{lem:measure-concentration} implies that there exists an integer $m_p$ such that the measure $\mu_p$ is highly concentrated on the set $\mathcal{S}_{m_p} = \{ m_p\bm{1} \} \cup \{ m_p\bm{1} + \bm{e}_j\mid 1\le j\le k\}$. Specifically, the measure outside $\mathcal{S}_{m_p}$ satisfies
\begin{align}\label{eq:m_p constriant}
    \sum_{\bm t \notin \mathcal{S}_{m_p}} \mu_p(\bm t) \ll_{k} \lambda^{\eta} = p^{-\eta/q} = p^{-k(q-1)/q} \le p^{-1-\frac{\varepsilon}{k+1}},
\end{align}
where in our context $q=\frac{k+{\varepsilon}/{(k-1)}}{k-1}$, and $\eta=\eta(k,q)=\min\Big(k+q-2,k(q-1),2q-1,q+1\Big)=k(q-1)$.

\textbf{3. Combining all prime information together.}
We define an integer $N = \prod_{p } p^{m_p'}$, where $m_p'=m_p$ if $p\in\probP(\bigcup_{i=1}^k A_i)$ and $m_p'=0$ otherwise. A tuple $(a_1, \dots, a_k) \in \Omega$ is called ``bad'' if there exists some prime $p>p_0$ such that
\[
(v_p(a_1),\dots,v_p(a_k))\notin \mathcal S_{m_p}.
\]

Let $\Omega_{\text{bad}}$ be the set of all bad tuples in $\Omega$. By \eqref{eq:m_p constriant}, 
\begin{align*}
\frac{|\Omega_{\text{bad}}|}{|\Omega|} &\le \frac{1}{|\Omega|}\sum_{p>p_0}\sum_{(a_1,\cdots,a_k)\in \Omega}\mathbf{1}_{{(v_p(a_1), \dots, v_p(a_k))\notin \mathcal{S}_{m_p}}}\\
&\le \sum_{p >p_0} {\sum_{\bm t \notin \mathcal{S}_{m_p}} \mu_p(\bm t)}
\ll_{k} \sum_{p >p_0} p^{-1-\frac{\varepsilon}{k+1}}\le \frac{1}{2},
\end{align*}
provided that $p_0=p_0(k,\varepsilon)$ is chosen sufficiently large.

Let $\Omega' = \Omega \setminus \Omega_{\text{bad}}$, so $|\Omega'| \ge \frac{1}{2}|\Omega|$. For any $(a_1, \dots, a_k) \in \Omega'$ and any prime $p\in \probP\!\left(\bigcup_{i=1}^k A_i\right)$, the valuation vector $\bm v_p = (v_p(a_1), \dots, v_p(a_k))$ must lie in $\mathcal{S}_{m_p}$. By the definition of $\mathcal{S}_{m_p}$, the valuations $v_p(a_i)$ are all equal to $m_p$, with at most one exception which equals $m_p+1$.
If $p\notin \probP\!\left(\bigcup_{i=1}^k A_i\right)
$, we clearly have $v_p(a_i)=0=m_p'=v_p(N)$ for all $i$. This completes the proof of Proposition \ref{prop:minimal-structure}.
\end{proof}

\subsection{Proof of Theorem \ref{thm:main}: The Combinatorial Argument and Final Contradiction}

We now arrive at the final step of the proof of Theorem \ref{thm:main}. We will show that any  sets exhibiting the structure derived in Proposition \ref{prop:minimal-structure} cannot, in fact, be a counterexample.

Proposition \ref{prop:minimal-structure} asserts the existence of an integer $N$ and a large subset $\Omega' \subset \Omega$ such that for any tuple $(a_1, \dots, a_k) \in \Omega'$, the quantities $a_i' := a_i/N$ are positive integers satisfying $\sum_{i=1}^k v_p(a_i') \le 1$ for all primes $p$. This immediately implies two simple but useful properties.

\begin{lemma}\label{lem:purity}
For every tuple $(a_1,\dots,a_k)\in \Omega'$, the integers $a_1',\dots,a_k'$ are square-free and pairwise coprime.
\end{lemma}
\begin{proof}
These properties follow directly from the $p$-adic constraint. For any prime $p$, if $p^2 | a_j'$ for some $j$, then $v_p(a_j') \ge 2$, which violates $\sum v_p(a_i') \le 1$.

    Suppose $p | a_i'$ and $p | a_j'$ for some $i \ne j$. This means $v_p(a_i') \ge 1$ and $v_p(a_j') \ge 1$. Then
    \( \sum_{l=1}^k v_p(a_l') \ge v_p(a_i') + v_p(a_j') \ge  2. \)
    This again violates the constraint that the sum cannot exceed 1.
\end{proof}

Now by Lemma \ref{lem:purity}, for any tuple $(a_1, \dots, a_k) \in \Omega'$, its greatest common divisor satisfies
\[
\gcd(a_1, \dots, a_k) = \gcd(Na_1', \dots, Na_k') = N\ge D.
\]
Clearly the map $(a_1, \dots, a_k) \mapsto  (a_1', \dots, a_k')$
is injective, and $a_i' \in [X_i/N, 2X_i/N]$ for each $i=1, \dots, k$.
So the size of $\Omega'$ is bounded by
\begin{align*}
|\Omega'| &\le |\{ (a_1', \dots, a_k') \in \mathbb{Z}^k \mid \forall i, X_i/N \le a_i' \le 2X_i/N \}| 
\le 2^k\frac{X_1 \cdots X_k}{D^k}.
\end{align*}
On the other hand, by Proposition \ref{prop:minimal-structure}, the lower bound is $|\Omega'| \ge \frac{1}{2}|\Omega| \ge \frac{\delta}{2}\prod_{i=1}^k |A_i|$.
Combining the two bounds on $|\Omega'|$ gives
\begin{equation}\label{eq:up bd}
\prod_{i=1}^k|A_i| \le \delta^{-1} 2^{k+1}\frac{X_1 \cdots X_k}{ D^k}.
\end{equation}
However, the definition of a minimal counterexample requires that the sets violate the bound stated in Theorem \ref{thm:main}, which together with \eqref{eq:up bd} implies
\[
 C_k \delta^{-\frac{k+\varepsilon/(k-1)}{k-1}} <\delta^{-1} 2^{k+1}.
\]
Recalling $C_k=2^{4k}>2^{k+1}$. This yields a contradiction.

Therefore, no minimal counterexample can exist, and Theorem \ref{thm:main} must hold true.

\subsection{Application: The Pairwise GCD Problem} \label{subsec:swise-gcd}
We conclude this section by giving an application of our main result Theorem \ref{thm:main} to the pairwise GCD problem.  Actually we turn more generally to the {$s$-wise GCD problem}, where every sub-tuple of size $s$ is required to have a large GCD.
This is done via a projection argument.

\begin{theorem}[Bound for the $s$-wise GCD Problem]\label{thm:pairwise_bound}
    Let $k \geq s \ge 2$. Let $X_1, \dots, X_k, D \in [1, \infty)$ with $D \leq \min(X_1, \dots, X_k)$, and let $\delta \in (0,1]$. Let $A_i \subset [X_i, 2X_i]$ for $1 \leq i \leq k$ be sets of integers. Suppose $\Omega \subset \prod_{i=1}^k A_i$ is a set of tuples with $|\Omega| \geq \delta \prod_{i=1}^k |A_i|$ such that for all $(a_1, \dots, a_k) \in \Omega$ and for every subset of indices $I = \{i_1, \dots, i_s\} \subset \{1, \dots, k\}$ of size $s$, we have
    \[
        \gcd(a_{i_1}, \dots, a_{i_s}) \ge D.
    \]
    Then for any $\varepsilon > 0$, we have
    \[
        \prod_{i=1}^k |A_i| \ll_{k,\varepsilon} \delta^{-\frac{k}{s-1}-\frac{k\varepsilon}{s}} \frac{\prod_{i=1}^k X_i}{D^k}.
    \]
\end{theorem}

\begin{proof}
 We derive the result by projecting the set $\Omega$ onto  $s$-dimensional subspaces.
Let $\mathcal{I}$ be the collection of all $\binom{k}{s}$ subsets of $\{1, \dots, k\}$ of size $s$. For any $I \in \mathcal{I}$, let
\[
 \pi_I:\prod_{i=1}^k A_i \to \prod_{i\in I} A_i
 \]
be the coordinate projection, and write $\Omega_I:=\pi_I(\Omega)$.

The size of $\Omega$ is bounded by 
    \(
    |\Omega| \le |\Omega_I| \prod_{j \notin I} |A_j|.
    \)
Using the hypothesis $|\Omega| \ge \delta \prod_{i=1}^k |A_i|$, we obtain
    \[
    |\Omega_I| \ge \frac{|\Omega|}{\prod_{j \notin I} |A_j|} \ge \delta \prod_{i \in I} |A_i|.
    \]
By assumption, every tuple in $\Omega$ satisfies the GCD condition on the coordinates in $I$. Thus, every tuple in the projection $\Omega_I$ satisfies $\gcd(\{a_i\}_{i \in I}) \ge D$.
We can therefore apply Theorem \ref{thm:main} (with the $s$-dimensional result) to each $\Omega_I$. This yields $\binom{k}{s}$ inequalities
\begin{equation}\label{eq:pairwise_bounds}
    \prod_{i \in I} |A_i| \ll_{\varepsilon} \delta^{-\frac{s}{s-1}-\varepsilon} \frac{\prod_{i \in I} X_i}{D^s}, \quad \text{for each } I \in \mathcal{I}.
\end{equation}
We multiply the inequalities over all $I \in \mathcal{I}$.
Observe that each index $i \in \{1, \dots, k\}$ belongs to exactly $\binom{k-1}{s-1}$ subsets in $\mathcal{I}$. Consequently, 
\[
   \left( \prod_{i=1}^k |A_i| \right)^{\binom{k-1}{s-1}}\ll \delta^{-\left(\frac{s}{s-1}+\varepsilon\right)\binom{k}{s}} D^{-s\binom{k}{s}} \left( \prod_{i=1}^k X_i \right)^{\binom{k-1}{s-1}},
\]
which gives
\begin{align*}
    \prod_{i=1}^k |A_i| \ll \delta^{-\frac{k}{s-1} - \frac{k\varepsilon}{s}} \frac{\prod_{i=1}^k X_i}{D^k},
\end{align*}
as desired.
\end{proof}
\begin{remark}
    The bound is sharp with respect to the parameters $X_i$ and $D$ (ignoring the $\delta$ factor). This is evident in the case $\delta=1$: by choosing each $A_i$ as the set of multiples of $D$, we strictly satisfy the condition $\gcd\ge D$ and achieve the product size $\prod |A_i| \asymp \frac{\prod X_i}{D^k}$.
\end{remark}

\section{An Alternative Approach via the Large Sieve-Type Inequality}
In this section, we establish two large sieve-type inequalities that may be of independent interest, and proves Theorem \ref{thm:k-dim-hybrid} and Theorem \ref{thm:lcm-dim-clean}.

\subsection{The Large Sieve-Type Inequality}
We first establish the following inequality.

\begin{lemma}
\label{lem:standard_sieve}
Let $\{\xi_n\}_{n\sim X}$ be any sequence of complex numbers. For any $X, D \ge 1$, we have
\begin{equation}
\label{lem_dual} \sum_{d\sim D} \Bigg| \sum_{\substack{n\sim X \\ d \mid n}} \xi_n \Bigg|^2 \ll \left(XD^{\varepsilon-1} + D \right) \sum_{n\sim X} |\xi_n|^2.
\end{equation}
\end{lemma}

\begin{proof}
By the duality principle, it suffices to prove 
\begin{equation}
    \label{eqn:large sieve} \sum_{n\sim X} \Bigg| \sum_{\substack{ d\sim D\\ d \mid n}} \eta_d \Bigg|^2 \ll \left( X D^{\varepsilon-1} + D \right) \sum_{d\sim D} |\eta_d|^2
\end{equation}
for any sequence of complex numbers $\{\eta_d\}_{d\sim D}$.
Let $S$ denote the sum on the left-hand side of \eqref{eqn:large sieve}. Expanding the square and interchanging the order of summation, we get
\begin{align*}
S &= \sum_{d_1, d_2 \sim D} \eta_{d_1} \overline{\eta_{d_2}} \sum_{\substack{n\sim X \\ [d_1, d_2] \mid n}} 1,
\end{align*}
where $[d_1, d_2]$ denotes the least common multiple of $d_1$ and $d_2$.
We can express the inner sum as $\frac{X}{[d_1, d_2]} + E(d_1, d_2)$, where the error term $|E(d_1, d_2)| \le 1$.
This splits $S$ into
\[ S = S_M + S_E, \]
where
\[ S_M = X \sum_{d_1, d_2 \sim D} \frac{\eta_{d_1} \overline{\eta_{d_2}}}{[d_1, d_2]} \quad \text{and} \quad S_E = \sum_{d_1, d_2 \sim D} \eta_{d_1} \overline{\eta_{d_2}} E(d_1, d_2). \]
We trivially bound the error term using the Cauchy-Schwarz inequality:
\begin{align*}
|S_E| &\le \sum_{d_1, d_2 \sim D} |\eta_{d_1}| |\overline{\eta_{d_2}}| = \left( \sum_{d\sim D} |\eta_{d}| \right)^2 \le D \sum_{d\sim D} |\eta_{d}|^2.
\end{align*}

For the main term, we use the identity $[d_1, d_2] = d_1 d_2 / \gcd(d_1, d_2)$.  We then apply the basic inequality $|\eta_{d_1} \overline{\eta_{d_2}}| \le \frac{1}{2}(|\eta_{d_1}|^2 + |\eta_{d_2}|^2)$. By symmetry, this cleanly is
\begin{equation}\label{eq:large_sieve_main_term}
    |S_M| \le X \sum_{d_1, d_2 \sim D} \frac{\gcd(d_1, d_2)}{d_1 d_2} |\eta_{d_1}|^2 = X \sum_{d_1 \sim D} |\eta_{d_1}|^2 \Bigg( \sum_{d_2 \sim D} \frac{\gcd(d_1, d_2)}{d_1 d_2} \Bigg).
\end{equation}
Denote the inner sum by $C(d_1)$. Using the identity $n=\sum_{d|n}\varphi(d)$, we get
\begin{align*}
   C(d_1) &= \frac{1}{d_1} \sum_{l \mid d_1} \varphi(l) \sum_{\substack{d_2 \sim D \\ l \mid d_2}} \frac{1}{d_2} = \frac{1}{d_1} \sum_{l \mid d_1} \frac{\varphi(l)}{l} \sum_{k \sim D/l} \frac{1}{k}
  \ll \frac{\tau(d_1)}{d_1}.
\end{align*}

For any fixed $\varepsilon > 0$, there exists a constant $C_\varepsilon$ such that for all integers $n \ge 1$, $\tau(n) \le C_\varepsilon n^\varepsilon$.
Applying this uniform bound for all $d_1 \in [D, 2D]$, we have
\begin{equation}\label{eq:ls_main_term_bound}
    |S_M| \ll X D^{\varepsilon-1} \sum_{d \sim D} |\eta_d|^2.
\end{equation}
Combining the bounds for $S_M$ and $S_E$ completes the proof.
\end{proof}

Now we establish the dual inequality, where the roles of divisors and multiples are interchanged.

\begin{lemma}[Dual Large Sieve Inequality]
\label{lem:dual_sieve}
Let $\{\xi_n\}_{n\sim X}$ be any sequence of complex numbers. For any $X, D \ge 1$, we have
\begin{equation}
    \label{eqn:dual_sieve}
    \sum_{d\sim D} \Bigg| \sum_{\substack{n\sim X \\ n \mid d}} \xi_n \Bigg|^2 \ll_{\varepsilon} \left( D X^{\varepsilon-1} + X \right) \sum_{n\sim X} |\xi_n|^2.
\end{equation}
\end{lemma}

\begin{proof}
Let $S$ denote the sum on the left-hand side. Expanding the square gives
\[
S = \sum_{d\sim D} \Bigg| \sum_{\substack{n\sim X \\ n \mid d}} \xi_n \Bigg|^2 = \sum_{n_1, n_2 \sim X} \xi_{n_1} \overline{\xi_{n_2}} \sum_{\substack{d\sim D \\ n_1 \mid d, \, n_2 \mid d}} 1.
\]
The inner sum counts the number of multiples of $\lcm(n_1, n_2)$ in the interval $[D, 2D]$. Then we have $S=S_M+S_E$ with the error term bounded trivially by the Cauchy-Schwarz inequality:
\[
|S_E| \le \sum_{n_1, n_2 \sim X} |\xi_{n_1}\overline{\xi_{n_2}}| = \left(\sum_{n\sim X} |\xi_n|\right)^2 \le X \sum_{n\sim X} |\xi_n|^2.
\]
The main term is 
\[
S_M = D \sum_{n_1, n_2 \sim X} \frac{\xi_{n_1} \overline{\xi_{n_2}}}{\lcm(n_1, n_2)} = D \sum_{n_1, n_2 \sim X} \frac{\gcd(n_1, n_2)}{n_1 n_2} \xi_{n_1} \overline{\xi_{n_2}}.
\]
We observe that this quadratic form is  identical to the main term in the proof of Lemma \ref{lem:standard_sieve} (see \eqref{eq:large_sieve_main_term}), but with the roles of the parameters and variables interchanged:
\(
X \leftrightarrow D,\  d_i \leftrightarrow n_i,\  \eta \leftrightarrow \xi.
\)
Therefore by \eqref{eq:ls_main_term_bound}, we immediately obtain the bound for our current main term:
\[
|S_M| \ll_{\varepsilon} D X^{\varepsilon-1} \sum_{n\sim X} |\xi_n|^2.
\]
Combining the bounds for $S_M$ and $S_E$. The proof is complete.
\end{proof}

\subsection{Proof of Theorem \ref{thm:k-dim-hybrid}}
In this subsection we first give the following proposition for dyadic block estimate, and then Theorem~\ref{thm:k-dim-hybrid} follows subsequently.

Let $X_1:=\min_{1\le i\le k}(X_i)$. For each dyadic parameter \(\Delta\in[D,2X_1]\), write
\[
S(\Delta):=\sum_{d\sim \Delta}\prod_{i=1}^k |A_{i,d}|.
\]

\begin{proposition}[Dyadic block estimate]\label{prop:unified-dyadic}
Let \(X_1,\Delta \) and \(S(\Delta)\) be as above.
Then
\begin{equation}\label{eq:small-block}
S(\Delta)
\ll_{k,\eta}
\Delta^{-(k-1)+2\eta}
\left(\prod_{i=1}^k X_i^{\frac{k-1}{k}}\right)
\left(\prod_{i=1}^k |A_i|^{1/k}\right).
\end{equation}
\end{proposition}
\begin{proof}
By H\"older's inequality,
\begin{equation}\label{eq:holder-block}
S(\Delta)
=
\sum_{d\sim \Delta}\prod_{i=1}^k |A_{i,d}|
\leq
\prod_{i=1}^k
\left(
\sum_{d\sim \Delta}|A_{i,d}|^k
\right)^{1/k}.
\end{equation}

We shall repeatedly use the following three estimates.

\noindent
   {\bf The \(L^\infty\)-bound.}
For each \(i\),
\begin{equation}\label{eq:Linfty}
\max_{d\sim \Delta}|A_{i,d}|
\leq
\max_{d\sim \Delta}\#\{n\in [X_i,2X_i]:d\mid n\}
\ll \frac{X_i}{\Delta}.
\end{equation}

\noindent    
    {\bf The \(L^1\)-bound.}
For each \(i\),
\begin{align}
\sum_{d\sim \Delta}|A_{i,d}|
&=
\sum_{d\sim \Delta}\sum_{\substack{n\in A_i\\ d\mid n}}1\le \sum_{n\in A_i}\tau(n)
\ll_{\eta}
|A_i|\,X_i^{\eta}.
\label{eq:L1}
\end{align}

\noindent    
    {\bf  The \(L^2\)-bound.}
For each $i$, if \(X_i\geq \Delta^{2-\eta}\) for some small $\eta>0$, then
Lemma~\ref{lem:standard_sieve} applied  with $\xi_n=\textbf{1}_{n\in A_i}$, giving
\begin{equation}\label{eq:L2-simplified}
\sum_{d\sim \Delta}|A_{i,d}|^2=\sum_{d\sim\Delta}\Bigg|\sum_{\substack{n\sim X_i\\d|n}}\textbf{1}_{n\in A_i}\Bigg|^2
\ll_{\eta}
X_i\Delta^{\eta-1}|A_i|.
\end{equation}

Assume first that $\Delta\le\sqrt{X_1}$. Note that \eqref{eq:L2-simplified} is available in this case. To estimate the $k$th moment in \eqref{eq:holder-block}, we extract $k-2$ maximum terms.  Combining \eqref{eq:Linfty} and \eqref{eq:L2-simplified}, we have
\begin{align}
\sum_{d\sim \Delta}|A_{i,d}|^k
&\leq
\left(\max_{d\sim \Delta}|A_{i,d}|\right)^{k-2}
\sum_{d\sim \Delta}|A_{i,d}|^2  \ll_{k,\eta} X_i^{k-1}\Delta^{-(k-1-\eta)}|A_i|.
\label{eq:kth-small}
\end{align}
Hence,
\[
S(\Delta)
\ll
\prod_{i=1}^k
\left( X_i^{k-1}\Delta^{-(k-1-\eta)}|A_i|
\right)^{1/k},
\]
which satisfies \eqref{eq:small-block}.

Assume now that \(\Delta\in [\sqrt{X_1},2X_1]\), and write
\[
I_0:=\{i:X_i<\Delta^2\},\quad I_1:=\{i:X_i\ge\Delta^2\},\quad r:=|I_0|.
\]
For \(i\in I_0\), we use \eqref{eq:Linfty} and \eqref{eq:L1} to get
\begin{align}
\sum_{d\sim \Delta}|A_{i,d}|^k
&\leq
\left(\max_{d\sim \Delta}|A_{i,d}|\right)^{k-1}
\sum_{d\sim \Delta}|A_{i,d}|
\ll_{k,\eta}
X_i^{k-1+\eta}\Delta^{-(k-1)}|A_i|\notag
\\
&\ll_{k,\eta} X_i^{k-1}\Delta^{-(k-1)+2\eta}|A_i|.\notag
\end{align}
For \(i\in I_1\), since \(X_i\geq \Delta^2\geq \Delta^{2-\eta}\), we use
\eqref{eq:Linfty} and \eqref{eq:L2-simplified} to get
\begin{align}
\sum_{d\sim \Delta}|A_{i,d}|^k
&\leq
\left(\max_{d\sim \Delta}|A_{i,d}|\right)^{k-2}
\sum_{d\sim \Delta}|A_{i,d}|^2 \ll_{k,\eta} X_i^{k-1}\Delta^{-(k-1-\eta)}|A_i|.
\notag
\end{align}
Substituting these two estimates above into
\eqref{eq:holder-block}, we obtain
\begin{align*}
S(\Delta)
&\ll_{k,\eta}
\prod_{i\in I_0}
\left( X_i^{k-1}\Delta^{-(k-1)+2\eta}|A_i|\right)^{1/k}
\prod_{i\in I_1}
\left(X_i^{k-1}\Delta^{-(k-1-\eta)}|A_i|\right)^{1/k}
 \\
&\ll_{k,\eta}
\Delta^{-(k-1)+\frac{k+r}{k}\eta}
\left(\prod_{i=1}^k X_i^{\frac{k-1}{k}}\right)
\left(\prod_{i=1}^k |A_i|^{1/k}\right),
\end{align*}
as desired.
\end{proof}
\begin{proof}[Proof of Theorem~\ref{thm:k-dim-hybrid}]
Let $\Omega$ be the set of tuples with $\gcd(a_1,\dots,a_k) \geq D$. By hypothesis,
\[
\delta \prod_{i=1}^k |A_i| \leq |\Omega| \le \sum_{d \geq D}\prod_{i=1}^k |A_{i,d}|.
\]
The sum is finite as $d$ is bounded above by
$2X_1$. We partition the range of summation for $d$ into dyadic intervals $[D_j, D_{j+1})$
with $D_j = 2^j D$, for $0 \leq j \leq J = \left\lfloor \log_2(2X_1/D) \right\rfloor$. The inequality becomes
\begin{equation}
\delta \prod_{i=1}^k |A_i| \leq \sum_{j=0}^J \sum_{d \sim D_j} \prod_{i=1}^k |A_{i,d}|
=
\sum_{j=0}^J S(D_j),
\end{equation}
by the definition of \(S(\Delta)\). 

Fix $\eta=(k-1)\varepsilon/(2k)$. 
Applying Proposition~\ref{prop:unified-dyadic}, we obtain
\begin{align*}
\delta\prod_{i=1}^k |A_i|
&\ll_{k,\eta}
\left(\prod_{i=1}^k X_i^{\frac{k-1}{k}}\right)
\left(\prod_{i=1}^k |A_i|^{1/k}\right)
\sum_{j=0}^J D_j^{-(k-1)+2\eta} \\
&\ll_{k,\eta}
D^{-(k-1)+2\eta}
\left(\prod_{i=1}^k X_i^{\frac{k-1}{k}}\right)
\left(\prod_{i=1}^k |A_i|^{1/k}\right),
\end{align*}
since \((k-1)-2\eta>0\), and
\(\sum_{j\ge0} 2^{-j((k-1)-2\eta)}\) converges.
Rearranging the inequality above gives \eqref{eq:hybrid-conditional}.
\end{proof}

\subsection{Proof of Theorem \ref{thm:lcm-dim-clean}}
For the proof of Theorem \ref{thm:lcm-dim-clean}, we note that instead of applying Lemma \ref{lem:dual_sieve}, we use a direct counting method which suffices to establish a stronger bound in the LCM case.

\begin{proof}[Proof of Theorem \ref{thm:lcm-dim-clean}]
Fix $\eta = \varepsilon/k$. For each integer $l \ge 1$, let $d_{A_i}(l) := \#\{a \in A_i : a \mid l\}$. 
Since every valid tuple $(a_1,\dots,a_k)$ has a least common multiple $l \le L$, we can bound the total number of such tuples by summing over all possible common multiples $l \le L$. By hypothesis, this yields
$$
\delta \prod_{i=1}^k |A_i| \le \sum_{l \le L} \prod_{i=1}^k d_{A_i}(l).
$$
Applying Hölder's inequality to the right-hand side, we obtain
\begin{equation}\label{eq:lcm-holder-clean}
\delta \prod_{i=1}^k |A_i| \le \prod_{i=1}^k \left( \sum_{l \le L} d_{A_i}(l)^k \right)^{1/k}.
\end{equation}

To estimate the $k$-th moment, we extract $k-1$ maximum terms. Using the standard divisor bound, we have $\max_{l \le L} d_{A_i}(l) \le \max_{l \le L} \tau(l) \ll_\eta L^\eta$. For the remaining first moment, 
\[\sum_{l\le L} d_{A_i}(l) = \sum_{n\in A_i} \#\{l\le L:\ n\mid l\}\ll \frac{L}{X_i}|A_i|.\]
Thus,
$$
\sum_{l \le L} d_{A_i}(l)^k \le \Bigl(\max_{l \le L} d_{A_i}(l)\Bigr)^{k-1} \sum_{l \le L} d_{A_i}(l) \ll_\eta  L^{1+(k-1)\eta} X_i^{-1} |A_i|.
$$

Substituting this bound back into \eqref{eq:lcm-holder-clean}, we get
$$
\delta \prod_{i=1}^k |A_i| \ll_{k,\eta} \prod_{i=1}^k \left( L^{1+(k-1)\eta} X_i^{-1} |A_i| \right)^{1/k} = L^{1+(k-1)\eta} \left(\prod_{i=1}^k X_i\right)^{-1/k} \left(\prod_{i=1}^k |A_i|\right)^{1/k},
$$
which is
$$
\prod_{i=1}^k |A_i| \ll_{k,\eta} \delta^{-\frac{k}{k-1}} \frac{L^{\frac{k}{k-1} + k\eta}}{\left(\prod_{i=1}^k X_i\right)^{\frac{1}{k-1}}}
$$
as desired.
\end{proof}

\section{Sharpness of GCD and LCM bounds}\label{sec:sharp}
In this section we explain why the main upper bounds obtained in Theorem \ref{thm:main} (Theorem \ref{thm:k-dim-hybrid}) and Theorem \ref{thm:lcm-dim-clean} are essentially best possible. 
The constructions below show that, in the natural non-trivial parameter ranges, the dependence on 
\(\delta\), \(D\), \(L\), and the  \(X_i\) cannot be improved, up to absolute constants  and $\varepsilon$-loss appearing in the theorem.

\subsection{Sharpness of the GCD bound}\label{subsec:sharp-gcd}
The case \(\delta=1\) is straightforward.
Let $D$ be a large integer. For each $i=1, \dots, k$, define $A_i$ to be the set of all multiples of $D$ within the interval $[X_i, 2X_i]$.
Clearly, $\gcd(a_1, \dots, a_k) \ge D$ for every tuple in $\prod_{i=1}^k A_i$. Then
\[
\prod_{i=1}^k |A_i| \asymp \prod_{i=1}^k \frac{X_i}{D} = \frac{\prod_{i=1}^k X_i}{D^k},
\]
which matches the bound in Theorem \ref{thm:main} for $\delta=1$.

Next, we analyze the general case $\delta < 1$. For simplicity, we assume $X_i = X$ and $A_i = A$ for all $i$.
Let $D \ge \delta^{-1/(k-1)}$ be given. We define $D_0 := \left\lfloor \delta^{\frac{1}{k-1}} D \right\rfloor$.
Let 
\[
A = \{a \in [X, 2X] : D_0 \mid a \}.
\]
It is easy to see that
\begin{align*}
\prod_{i=1}^k |A_i| \asymp \left(\frac{X}{D_0}\right)^k &\asymp \left(\frac{X}{\delta^{1/(k-1)} D}\right)^k  = \delta^{-\frac{k}{k-1}} \frac{X^k}{D^k}.
\end{align*}
This matches the upper bound. It remains to verify that the proportion of tuples with $\gcd \ge D$ is approximately $\delta$.

For any tuple $(a_1, \dots, a_k) \in A^k$, we can write $a_i = D_0 m_i$ with $m_i\sim X/D_0$. The condition $\gcd(a_1, \dots, a_k) \ge D$ is equivalent to
$
\gcd(m_1, \dots, m_k) \ge D/D_0.
$
Provided that the range $X/D_0$ is sufficiently large compared to $D/D_0$ so that the distribution of the tuples $(m_1, \dots, m_k)$ is sufficiently uniform, the proportion of integer tuples with $\gcd \ge D/D_0$ is given by
\[
\mathbb{P}\left(\gcd(m_1, \dots, m_k) \ge D/D_0\right)  \asymp \sum_{g \ge D/D_0} \frac{1}{\zeta(k) g^k} 
\asymp \frac{1}{(D/D_0)^{k-1}} \asymp \delta.
\]
Thus, this construction proves that the bound in Theorem \ref{thm:main} is attainable.

\begin{remark}
    The argument above relies on the asymptotic density of tuples with a given GCD. For a fully rigorous construction that guarantees the proportion is strictly at least $\delta$, we refer the reader to see the method used in the next subsection. A similar construction can be readily adapted to the GCD case.
\end{remark}

\subsection{Sharpness of the LCM bound}\label{subsec:sharp-lcm}
In the final subsection, we show that the bound obtained in Theorem \ref{thm:lcm-dim-clean} is essentially sharp. 

\begin{proposition}[Sharpness for the LCM bound]
\label{prop:lcm-sharp-rigorous}
Let \(k\ge 2\), let \(\delta\in(0,1]\). Assume that
\(
M:=\Bigl\lceil \,c_k\delta^{-1/(k-1)}\Bigr\rceil
\) and
 \( Q:=C_k\left(L^{-1}{\prod_{i=1}^k X_i}\right)^{1/(k-1)}
\)
satisfies
\[
M\log(2M)\ll_k Q \ll_k \min_{1\le i\le k} X_i,
\]
for suitable constants \(c_k,C_k>0\) depending only on \(k\).
Then there exist sets \(A_i\subset [X_i,2X_i]\) such that
\(
\#\Bigl\{(a_1,\dots,a_k)\in \prod_{i=1}^k A_i:\ \operatorname{lcm}(a_1,\dots,a_k)\le L\Bigr\}
\ge
\delta \prod_{i=1}^k |A_i|,
\)
and
\[
\prod_{i=1}^k |A_i|
\gg_k
\delta^{-\frac{k}{k-1}}
\frac{L^{\frac{k}{k-1}}}{\bigl(\prod_{i=1}^k X_i\bigr)^{1/(k-1)}}.
\]
\end{proposition}

\begin{proof}
Choose \(M\) distinct primes
\(
q_1,\dots,q_M \in [Q,2Q].
\) This is possible since we are assumed \(Q\gg_k M\log(2M)\).
For each $i \in \{1, \dots, k\}$, we construct the set $A_i$ as the union of $M$ subsets
\[
    A_i = \bigcup_{j=1}^M A_i^{(j)}, \quad \text{where } A_i^{(j)} = \{ n \in [X_i, 2X_i] : q_j \mid n \}.
\]
Since \(q_j\sim Q\), we have, say
\[
\frac{X_i}{4Q}\le|A_i^{(j)}|\le \frac{2X_i}{Q}.
\]
Moreover, for \(j\neq \ell\),
\[
|A_i^{(j)}\cap A_i^{(\ell)}|
\le \frac{2X_i}{Q^2}.
\]
As \(Q\gg M\), inclusion-exclusion gives
\[
    \sum_{j=1}^M|A_i^{(j)}|\ge |A_i| \ge \sum_{j=1}^M|A_i^{(j)}|-\sum_{1\le j< l\le M}|A_i^{(j)}\cap A_i^{(l)}|\ge M \frac{X_i}{4Q} -\binom{M}{2}\frac{2X_i}{Q^2},
\]
which implies \(2MX_i/Q\ge|A_i|\ge M{X_i}/(8{Q})\) by choosing constant $C_k$ large enough.
Hence
\[
(2M)^k \frac{\prod_{i=1}^k X_i}{Q^k}\ge\prod_{i=1}^k |A_i|
\ge
M^k \frac{\prod_{i=1}^k X_i}{(8Q)^k}
\gg_k
\delta^{-\frac{k}{k-1}}
\frac{L^{\frac{k}{k-1}}}{\bigl(\prod_{i=1}^k X_i\bigr)^{1/(k-1)}}.
\]

Consider now the ``good'' tuples $(a_1, \dots, a_k)\in\prod_{i=1}^kA_i$ for which there exists some $j$ such that $a_i \in A_i^{(j)}$ for all $1\le i\le k$.
For such a ``good'' tuple, we have $a_i = q_j m_i$ with $m_i \le 2X_i/Q$. Then the least common multiple satisfies
\[
\operatorname{lcm}(a_1,\dots,a_k)
\le q_j \prod_{i=1}^k m_i
\le
2Q \prod_{i=1}^k \frac{2X_i}{Q}\le \frac{2^{k+1}}{C_k^{k-1}}L
\le L,
\]
by choosing $C_k$  large enough. 

If \(M=1\), then \(A_i=A_i^{(1)}\) for every \(i\), so every tuple in \(\prod_{i=1}^k A_i\) is good, and the desired density condition is immediate. Hence we may assume \(M\ge 2\).
By inclusion-exclusion principle again, we have the number of  ``good" tuples is at least
\begin{align*}
\left|\bigcup_{j=1}^M \prod_{i=1}^k A_i^{(j)}\right|
&\ge \sum_{j=1}^M \prod_{i=1}^k |A_i^{(j)}|
- \sum_{1\le j<\ell\le M} \left|\prod_{i=1}^k (A_i^{(j)}\cap A_i^{(\ell)})\right|\\
&\ge \frac{M}{(4Q)^k} \prod_{i=1}^k X_i-\binom{M}{2}\prod_{i=1}^k \frac{2X_i}{Q^2} \ge \frac{1}{2}\frac{M}{(4Q)^k} \prod_{i=1}^k X_i,
\end{align*}
by choosing $C_k$  large enough again.

The density of these ``good" tuples within $\prod_{i=1}^k A_i$ is therefore
\[
   \ge \frac{M (\prod_{i=1}^k X_i) / (4Q)^k}{2(2M)^k(\prod_{i=1}^kX_i) / Q^k} =  \frac{1}{2\cdot 8^kM^{k-1}}\ge \delta,
\]
where the last inequality is ensured by choosing suitable $c_k$. This confirms that the condition of having at least a proportion $\delta$ of tuples satisfying $\lcm \le L$ is met, and hence we are done.
\end{proof}


\section*{Acknowledgements}
The author would like to thank Professor Andrew Granville for his guidance and many helpful discussions and comments. The author also gratefully acknowledges the China Scholarship Council (CSC) for financial support, and the D\'epartement de math\'ematiques et de statistique, Universit\'e de Montr\'eal for providing an excellent research environment where this work was carried out.

\end{document}